\documentclass[a4paper,12pt]{amsart}
\usepackage{a4wide}
\usepackage[right=2.5cm,left=2.5cm,top=3cm,bottom=3.5cm]{geometry}
\usepackage[english]{babel}
\usepackage{amssymb,amsmath,amsthm,amsfonts,xcolor,cite,pifont}
\usepackage[colorlinks]{hyperref}
\usepackage[showonlyrefs]{mathtools}
\mathtoolsset{showonlyrefs=true}
\usepackage[normalem]{ulem}

\usepackage{bm}

\numberwithin{equation}{section}

\newtheorem{theorem}{Theorem}[section]
\newtheorem{lemma}[theorem]{Lemma}
\newtheorem{proposition}[theorem]{Proposition}

\newtheorem{corollary}[theorem]{Corollary}

\theoremstyle{definition}
\newtheorem{definition}[theorem]{Definition}
\newtheorem{remark}[theorem]{Remark}

\newcommand{\R}{\mathbb R}

\newcommand{\C}{{\rm cap}}

\newcommand{\Dzc}{\mathcal{D}^{1,2}_z(\R^{n+1}) }
\newcommand{\Rnn}{\mathbb R^{n+1}}

\newcommand{\sub}{\subseteq}								
\newcommand{\abs}[1]{\left| #1 \right|}						
\newcommand{\norm}[1]{\left\lVert #1 \right\lVert}			

\newcommand{\Ds}{\mathcal{D}^{s,2}(\R^n)}					
\newcommand{\DsO}{\mathcal{D}^{s,2}(\R^n\setminus\Omega)}	
\newcommand{\Dext}{\mathcal{D}^{1,2}_z(\overline{\R^{n+1}_+})}	
\newcommand{\DextO}{\mathcal{D}^{1,2}_z(\overline{\R^{n+1}_+}\setminus\Omega)}	
\newcommand{\dive}{\mathrm{div}}			

\newenvironment{bvp}{\left\{\begin{aligned}  }{\end{aligned}\right.}	

\DeclareMathOperator{\Tr}{\mathrm{Tr}}		

\def\A{\mathcal A}
\def\C{{\rm cap}}
\def\Cs{{\rm cap}_s}

\def\eps{\varepsilon}

\def\ges{\gtrsim}

\author{Eleonora Cinti}
\address{Dipartimento di Matematica, Universit\`a di Bologna, 40126 Bologna, Italy}
\email{eleonora.cinti5@unibo.it}

\author{Roberto Ognibene}
\address{Dipartimento di Ingegneria meccanica \\ energetica, gestionale e dei trasporti \\
	Università di Genova \\  16145  Genova \\ Italy} 
\email{roberto.ognibene@edu.unige.it} 

\author{Berardo Ruffini}
\address{Dipartimento di Matematica, Universit\`a di Bologna, 40126 Bologna, Italy} 
\email{berardo.ruffini@unibo.it}

\title{A quantitative stability inequality for fractional capacities}



\begin{document}

\begin{abstract}
The aim of this work is to show a non-sharp quantitative stability version of the fractional isocapacitary inequality. In particular, we provide a lower bound for the isocapacitary deficit in terms of the Fraenkel asymmetry. In addition, we provide the asymptotic behaviour of the $s$-fractional capacity when $s$ goes to $1$ and the stability of our estimate with respect to the parameter $s$.
\end{abstract}

\maketitle
\tableofcontents
\section{Introduction}

The classical isocapacitary inequality states that among sets which share the same amount of Lebesgue measure, balls minimize the electrostatic (Newtonian) capacity, that is, for any measurable set $\Omega$ with finite measure, the following scale invariant inequality holds true
\begin{equation}\label{eq:iso_0}
	|\Omega|^{(2-n)/n}\C(\Omega)\geq  |B|^{(2-n)/n}\C(B).
\end{equation}
Here $|\cdot|$ stands for the $n-$dimensional Lebesgue measure, $n\ge3$, $B$ is any ball in $\R^n$ and $\C(\cdot)$ is the standard electrostatic capacity in $\R^n$, defined for compact sets as
\begin{equation}\label{eq:def_std_cap}
	\C(\Omega):=\inf\left\{\int_{\R^n} \abs{\nabla u}^2\,dx\colon ~u\in C_c^\infty(\R^n),~u\geq 1 ~\text{in }\Omega \right\}.
\end{equation}
We observe that \eqref{eq:iso_0} can be rephrased in terms of the \emph{isocapacitary deficit}, by saying that
\begin{equation}\label{eq:isocapacitary}
	d_{\rm cap}(\Omega):=\frac{|\Omega|^{(2-n)/n}\C(\Omega)}{|B|^{(2-n)/n}\C(B)}-1 \geq 0.
\end{equation}
It is well known that the isocapacitary inequality is rigid, in the sense that $d_{\rm cap}(\Omega)$ vanishes if and only if $\Omega$ is equivalent to a ball up to a set of null {Lebesgue measure}.  Thus, it appears as a natural quest the attempt of  obtaining a quantitative stability version of \eqref{eq:isocapacitary}.  There are several possible geometric quantities that can properly  measure  the difference between a generic set and a ball with the same volume.  The most natural one is the so-called  \emph{Fraenkel asymmetry}, first proposed by L. E. Fraenkel  given by
\[
\A(\Omega)=\inf\left\{\frac{|\Omega\Delta B|}{|\Omega|}\,:\, \text{ $B$ is a ball with $|B|=|\Omega|$}  \right\}.
\]
The first attempts in this direction were made in the '90s. In particular, in \cite{HHW} stability inequalities of the form
\begin{equation*}\label{eq:non_sharp}
	d_{\rm cap}(\Omega)\geq C_n \A(\Omega)^{n+1}
\end{equation*}
were proved, restricting to the class of convex sets when $n\geq 3$\footnote{In the planar case the suitable capacity is the logarithmic capacity.}. Nevertheless, in \cite{HHW} the optimal exponent was conjectured to be 2, that is
\begin{equation}\label{eq:sharp}
	d_{\rm cap}(\Omega)\geq C_n' \A(\Omega)^2,
\end{equation}
which is asymptotically sharp for small asymmetries. Inequality \eqref{eq:sharp} was proved in the planar case in \cite[Corollary 2]{HN1992} (see also \cite{HN} and \cite{HN1994} for related results with other notions of deficiencies). As far as higher dimensions are concerned, \eqref{eq:sharp} was proved by Fraenkel in \cite{Fraenkel2008} for starshaped sets, while in \cite{FMP} the authors provided the inequality \eqref{eq:sharp} with a suboptimal exponent but for general sets, i.e.
\begin{equation}\label{eq:quantitativefmp}
	d_{\rm cap}(\Omega)\ge C_n''\A(\Omega)^4.
\end{equation}
The conjecture in its full generality was finally established in \cite{DMM}. It is worth stressing that to get this result the authors need to exploit the suboptimal inequality \eqref{eq:quantitativefmp}. We finally mention \cite{Mukoseeva2020}, where the author treated the case of the $p$-capacity and proved the corresponding sharp inequality. We point out that the approach followed in \cite{DMM,Mukoseeva2020}, while leading to the sharp exponent $2$, does not allow to work out the explicit constant which multiplies the asymmetry. On the contrary, inequality \eqref{eq:quantitativefmp} is not sharp for small values of $\mathcal{A}(\Omega)$ but comes with an explicit constant $C_n''>0$. 

In this work we tackle the problem of quantification of the isocapacitary inequality in the fractional framework. 

Let $s\in(0,1)$ and let $n>2s$. We consider the fractional generalization of the capacity, defined for compact sets as follows
\begin{equation}\label{eq:s_cap_comp}
\Cs(\Omega)=\inf \left\{ [u]^2_{s}\colon u\in C^\infty_c(\R^n),~u\geq 1~\text{on }\Omega\right\},
\end{equation}
where
\begin{equation}\label{eq:seminorm}
	[u]_s:=\left( \int_{\R^{2n}}\frac{\abs{u(x)-u(y)}^2}{\abs{x-y}^{n+2s}}\,dx\,dy \right)^{\frac{1}{2}}
\end{equation}
denotes the fractional Gagliardo seminorm of order $s$. The definition of fractional capacity of a general closed set $\Omega\sub\R^n$ is given in Definition \ref{def:capacity_s}, which can be easily proved to be equivalent to \eqref{eq:s_cap_comp} when $\Omega$ is compact.
As a straightforward consequence of the {fractional analogue of the P\'olya-Szeg\"o inequality (proved in \cite[Theorem 9.2]{Almgren1989}, see also Proposition \ref{polyaszego} below for an ``extended'' version)} one can easily derive the fractional isocapacitary inequality, stating that
\begin{equation}\label{eq:frac_iso}
	|\Omega|^{(2s-n)/n}\Cs(\Omega)\geq  |B|^{(2s-n)/n}\Cs(B),
\end{equation}
for any closed $\Omega\sub\R^n$ with finite measure and for any closed ball $B$. The aim of this work is to quantify the fractional isocapacitary deficit
\begin{equation}\label{eq:frac_deficit}
	d_{{\rm cap}_s}(\Omega):=\frac{|\Omega|^{(2s-n)/n}\Cs(\Omega)}{|B|^{(2s-n)/n}\Cs(B)}-1
\end{equation}
in terms of the asymmetry of $\Omega$. We point out that, in view of the scaling properties of the fractional capacity, the term
\[
	|B|^{(2s-n)/n}\Cs(B)
\]
is a universal constant, not depending on the choice of the ball $B$. It is also worth remarking that $\Cs(\cdot)$ can be defined, through \eqref{eq:s_cap_comp}, on open sets $\mathcal{O}$ and its value coincide with $\Cs(\overline{\mathcal{O}})$.

{The fractional P\"olya-Szeg\"o inequality entails the rigidity of inequality \eqref{eq:frac_iso}, in the sense that the equality holds if and only if $\Omega$ is a ball in $\R^n$, see \cite[Theorem A.1]{FS2008}.} We now present our main result, which amounts to a quantitative stability inequality for the fractional capacity.
\begin{theorem}\label{main}
Let $s\in(0,1)$ and $n> 2s$. There exists a constant $C_{n,s}>0$, depending only on $n$ and $s$, such that for any closed set $\Omega\subset \R^n$ with finite measure, there holds 
\begin{equation}\label{quantitativa}
d_{{\rm cap}_s}(\Omega)=\frac{|\Omega|^{(2s-n)/n}\Cs(\Omega)}{|B|^{(2s-n)/n}\Cs(B)}-1\ge C_{n,s}\,\A^{\frac{3}{s}}(\Omega).
\end{equation}
Moreover the constant $C_{n,s}$ can be explicitly computed, see Remark \ref{rmk:asympt_1}.
\end{theorem}

{Our second result investigates the asymptotic behaviour of the function $s\mapsto\Cs(\Omega)$ when $s\to 1^-$, for a compact set $\Omega\sub\R^n$. In particular, we obtain that a suitable normalization of $\Cs$ behaves like the standard capacity as $s\to 1^-$ (see \eqref{eq:def_cap} for the precise definition of the classical notion of capacity).
}
\begin{proposition}\label{main2}
Let $n\ge 3$, then for every $\Omega\subset\mathbb{R}^n$ compact set, we have 
	\begin{equation}\label{eq:asympt_s_th1}
		\limsup_{s\nearrow 1} (1-s)\,{\rm cap}_s(\Omega)\le \frac{\omega_n}{2}\, {\rm cap}(\Omega),
	\end{equation}
	where $\omega_n:=|B_1|$ and $B_1$ denotes the unitary ball in $\R^n$. If in addition $\Omega$ is the closure of an open bounded set with Lipschitz boundary then
	\begin{equation}\label{eq:asympt_s_th2}
		\lim_{s\nearrow 1} (1-s)\,{\rm cap}_s(\Omega)=\frac{\omega_n}{2}\,{\rm cap}(\Omega).
	\end{equation}
\end{proposition}

{We observe that the exponent $3/s$ appearing in \eqref{quantitativa} is likely not sharp,\footnote{The optimal exponent was conjectured to be 2 in the fractional case as well.}. Nevertheless, since the constant $C_{n,s}$ in \eqref{quantitativa} can be explicitly computed (see Remark \ref{rmk:asympt_1}), together with its limit as $s\to 1^-$ (see Remark \ref{rmk:asympt_2}), our result entails an improvement of \eqref{eq:quantitativefmp}, asymptotically as $s\to 1^-$. In particular, thanks also to Theorem \ref{main2} and Lemma \ref{lemma:compact}, we are able to state the following.}


\begin{corollary}\label{cor}
	Let $n\geq 3$ and $\Omega\sub\R^n$ be a closed set with finite measure. Then \eqref{quantitativa} is stable as $s\to 1^-$ and there holds
	\[
	d_{\rm cap}(\Omega)\ge C_n\,\mathcal A(\Omega)^3,
	\]
	for some $C_n>0$ depending only on $n$, {whose explicit value can be found in Remark \ref{rmk:asympt_2}}.
\end{corollary}

\subsection{Strategy of the proof of Theorem \ref{main}}
Our proof is  inspired by that in \cite{BCV} where the authors deal with the (non-sharp) quantitative stability of the first eigenvalue of the fractional Laplacian with homogeneous Dirichlet exterior conditions. Such a result, in turn, relies on ideas established in \cite{FMP,HN}. Here, we provide a sketch of these arguments, starting with the classical case and then trying to emphasize the differences occurring in the fractional framework.

It is well known that, for any closed $\Omega\sub\R^n$ with finite measure, there exists a unique function $0\leq u_\Omega\leq 1$, belonging to a suitable functional space, that achieves $\C(\Omega)$. Such a function is called the \emph{capacitary potential} of $\Omega$. First, by means of the coarea formula one gets
\[
\C(\Omega)\sim \int_{\R^n}|\nabla u_\Omega|^2\,dx\sim \int_0^1 \left(\int_{\{u_\Omega=t\}}|\nabla u_\Omega|\,d\mathcal H^{n-1}\right)\,d t.
\]
The right-hand side of the latter equality, after some manipulation can be  written in terms of the perimeter of the superlevel sets $\{u_\Omega\geq t\}$, i.e.
\[
\C(\Omega)\sim \int_0^1 P(\{u_\Omega\ge t\})^2 f_1(t)\,dt.
\]
being $f_1$ a suitably chosen real  function depending only on the measure of $\{u_\Omega\ge t\}$. The  idea is then to exploit the sharp quantitative isoperimetric inequality \cite{FMP08,FiMP,CL} 
\[
d_{\rm Per}(E)\sim|E|^{(n-1)/n}P(E)-|B|^{(n-1)/n}P(B)\ges  \A(E)^2,
\] 
holding for any $E\sub\R^n$ in a suitable class and for any ball $B\sub\R^n$. Plugging this into the previous estimate, after some further manipulation, one gets 
\[
d_{\C}(\Omega)\ges \int_0^1 \A(\{u\ge t\})^2f_2(t)\,dt ,
\]
where  $f_2$ is an explicit {positive} integrable function depending only on the size of the superlevels of $u_\Omega$.  
Then we  reason  in  the spirit of \cite{HN}:

: heuristically, as long as $t$ is close to $\|u_\Omega\|_{L^\infty(\R^n)}=1$ we expect the set $\{ u_\Omega\ge t\}$ is close to $\Omega$ in $L^1$ and that $\A(\{ u_\Omega\ge t\})\sim\A(\Omega)$, and then, the idea is to seek for a threshold $T$
such that at once
\begin{itemize}
\item $\A(\{u\ge t\}) \sim \A(\Omega)$  as long as $t\in(T,1)$ and
\item the quantity $\int_T^1 f_2(t)\, dt$
results proportional to a power of $\A(\Omega)$.
\end{itemize}
 The previous two properties lead directly to the sought inequality. A bit more precisely, if the threshold $T$ is such that $1-T\gtrsim \mathcal A(\Omega)$, then the above strategy works, while if $1-T\lesssim \mathcal A(\Omega)$, then the fact that the asymmetry is large (with respect to $1-T$)
allows, by a simple comparison argument, to get an asymptotically stronger inequality of the form 
 \[
 d_{{\rm cap}}(\Omega)\ges \A(\Omega).
 \]
In the fractional case the existence of a capacitary potential $u_\Omega$ is guaranteed as well, see Remark \ref{rmk:potential}. However, the arguments described above cannot be directly implemented in the fractional scenario, due to nonlocal effects. Indeed the very first step fails, since a suitable coarea formula for non-integer Sobolev spaces is missing. A way to overcome these difficulties is provided by the so called Caffarelli-Silvestre extension for functions in fractional Sobolev spaces. Loosely speaking, this tool allows us to interpret nonlocal energies of functions defined on $\R^n$ as local energies of functions depending on one more variable. Namely, one can prove a characterization of the $s$-capacity in the fashion of
\begin{equation}\label{eq:intr_min}
\Cs(\Omega)\sim\inf\left\{ \int_{\R^{n+1}_+}z^{1-2s} |\nabla U(x,z)|^2\,dx\,dz \colon U(x,0)=u_\Omega(x)  \right\},
\end{equation}
where
\[
	\R^{n+1}_+:=\{(x,z)\colon x\in\R^n,~z>0\}
\]
and $U$ varies in a suitable functional space on $\R^{n+1}_+$. Moreover, one can prove that the infimum in \eqref{eq:intr_min} is uniquely achieved by a function $0\leq U_\Omega\leq 1$. We refer to Section \ref{sec:extd} for the precise setting and definitions. At this point, the above strategy may be applied on every horizontal slice $\{(x,z)\,:\,x\in\R^n\}$ and with $U_\Omega(\cdot,z)$ in place of $u_\Omega$. This way, we end up with
\[
d_{\Cs}(\Omega)\ges  \int_0^\infty z^{1-2s} \int_0^1 \A (\left\{  U(\cdot,z)\geq t  \right\})^2 f_z(t)\,dt\,dz
\] 
where, again, $f_z$ is an explicit real-valued function depending on the measure of the superlevels of $U(\cdot,z)$. Here it appears evident the extra inconvenience due to the presence of the integral in the $z-$variable. To get rid of this latter problem, we adapt ideas in \cite{BCV} to show the existence of a \emph{good} interval $(0,z_0)$, for which the (asymmetries of the) superlevels of $U_\Omega(\cdot,z)$ are close to (those of) the superlevels of $u_\Omega$, leading to
\[
\begin{aligned}
d_{\Cs}(\Omega)&\ges  \int_0^{z_0} z^{1-2s} \int_0^1 \A (\left\{   U_\Omega(\cdot,z)\geq t  \right\})^2 f_z(t)\,dt\,dz\\
&\sim \int_0^1 \A (\left\{ u_\Omega\geq t  \right\})^2 f_z(t)\,dt
\end{aligned}
\] 
and hence conclude similarly as in the classical local case.

\section{Preliminaries}\label{sec:pre}

In this section we introduce some prerequisites that are necessary in order to prove our main result.

\subsection{The fractional capacity}
First of all, we precisely define the functional setting we work in. For any open set $\mathcal{O}\subseteq\R^n$, we consider the homogeneous fractional Sobolev space $\mathcal{D}^{s,2}(\mathcal{O})$, defined as the completion of $C_c^\infty(\mathcal{O})$ with respect to the Gagliardo seminorm of order $s$
\begin{equation*}
	[u]_s= \left( \int_{\R^{2n}}\frac{\abs{u(x)-u(y)}^2}{\abs{x-y}^{n+2s}}\,dx\,dy \right)^{\frac{1}{2}}.
\end{equation*}
The space $\mathcal{D}^{s,2}(\mathcal{O})$ is an Hilbert space, naturally endowed with the following scalar product
\begin{equation*}
(u,v)_{\Ds}:=\int_{\R^{2n}}\frac{(u(x)-u(y))(v(x)-v(y))}{\abs{x-y}^{n+2s}}\,dx\,dy.
\end{equation*}
Moreover, by trivial extension we have that $\mathcal{D}^{s,2}(\mathcal{O})$ is continuously embedded in $\Ds$. We refer to \cite{Brasco2019} and \cite{Brasco2021} for more details concerning fractional homogeneous spaces and their characterizations. We limit ourselves to recall  the fractional Sobolev inequality, which reads as follows:
\begin{equation}\label{eq:sobolev_ineq}
S_{n,s}\norm{u}_{L^{2^*_s}(\R^n)}^2\leq [u]_s^2\quad\text{for all }u\in\Ds,
\end{equation}
where
\[
	2^*_s:=\frac{2n}{n-2s}
\]
denotes the critical Sobolev exponent in the fractional framework and $S_{n,s}>0$ denotes the best constant in the inequality. 
In particular, this result ensures the continuity of the embedding
\begin{equation*}
\Ds\hookrightarrow L^{2^*_s}(\R^n)
\end{equation*}
and it provides the following characterization
\begin{equation}\label{eq:charact_space}
\Ds=\{ u\in L^{2^*_s}(\R^n)\colon [u]_s<\infty \}. 
\end{equation}
We refer to \cite[Theorem 1.1]{Cotsiolis2004} (see also \cite[Theorem 7]{Savin2011} in the Appendix) and to \cite[Theorem 3.1]{Brasco2021} for the proofs of \eqref{eq:sobolev_ineq} and \eqref{eq:charact_space}, respectively.

We now introduce the definition of fractional capacity of a closed subset of $\R^n$.

\begin{definition}\label{def:capacity_s}
	Let $\Omega\sub\R^n$ be closed and let $\eta_\Omega\in C_c^\infty(\R^n)$ be such that $\eta_\Omega=1$ in an open neighbourhood of $\Omega$. We define the \emph{fractional capacity} of order $s$ (or \emph{$s$-capacity}) of the set $\Omega$ as follows:
	\begin{equation*}
		\Cs(\Omega):=\inf \{ [u]_s^2\colon u\in\Ds,~u-\eta_\Omega\in \mathcal{D}^{s,2}(\R^n\setminus \Omega) \}.
	\end{equation*}
\end{definition}
First of all, we point out that the above definition does not depend on the choice of the cut-off function $\eta_\Omega$. Indeed, if $\tilde{\eta}_\Omega\in C_c^\infty(\R^n)$ satisfies $\tilde{\eta}_\Omega=1$ in a neighbourhood of $\Omega$, then trivially
\[
	\eta_\Omega+\DsO=\tilde{\eta}_\Omega+\DsO.
\]
We also observe that, if $\Omega\sub\R^n$ is a compact set, then, by a simple regularization argument, one can easily prove that
\[
	\Cs(\Omega)=\inf\{[u]_s^2\colon u\in C_c^\infty(\R^n), ~u\geq 1 ~\text{in }\Omega\}.
\]
We also point out that it is not restrictive to assume that the admissible competitors $u$ in the definition of $\Cs(\Omega)$ satisfy
\[
	0\leq u\leq 1,\quad\text{a.e. in }\R^n,
\]
since
\[
	[u^+\wedge 1]_s\leq [u]_s\quad\text{for all }u\in\Ds,
\]
where $u^+$ denotes the positive part of $u$ and $a\wedge b=\min\{a,b\}$. We refer to lemmas 2.6 and 2.7 in \cite{Warma2015} for the proofs.
\begin{remark}\label{rmk:potential}
	By direct methods of the calculus of variations, it is easy to check that $\Cs(\Omega)$ is uniquely achieved (when $\Cs(\Omega)<\infty$) by a function $u\in \Ds$ such that $u-\eta_\Omega\in \DsO$. Hereafter, we denote such function by $u_\Omega$ and we call it the $s$-\emph{capacitary potential} (or simply the \emph{capacitary potential}) associated to $\Omega$. Moreover, it is easy to observe that $0\leq u_\Omega\leq 1$ a.e. in $\R^n$ and that $u_\Omega$ satisfies a variational equation, that is
\begin{equation*}
	(u_\Omega,\varphi)_{\Ds}=0,\quad\text{for all }\varphi\in\DsO.
\end{equation*}
\end{remark}
The notion of $s$-capacity of a set is in relation with the fractional Laplace operator of order $s$, which is defined, for $u\in C_c^\infty(\R^n)$, as follows
 \begin{align*}
(-\Delta)^s u (x):&=2\,\mathrm{P.V.}\,\int_{\R^n}
\frac{u(x)-u(y)}{\abs{x-y}^{n+2s}}\,dy
\\
&=2\lim_{r\to 0^+}
\int_{\{\abs{x-y}>r\}}\frac{u(x)-u(y)}{\abs{x-y}^{n+2s}}\,dy,
\end{align*}
where $\mathrm{P.V.}$ means that the integral has to be seen in the principal value sense. It is natural to extend the definition of fractional Laplacian applied to any function in $\Ds$, in a distributional sense. More precisely, given $u\in\Ds$, we have that $(-\Delta)^s u\in (\Ds)^*$ (with $(\Ds)^*$ denoting the dual of $\Ds$) and it acts as follows
\[
	_{(\Ds)^*}\langle (-\Delta)^s u,v\rangle_{\Ds}=(u,v)_{\Ds},\quad\text{for all }v\in\Ds.
\]
Therefore, in view of Remark \ref{rmk:potential}, we can say that the capacitary potential $u_\Omega\in\Ds$ weakly satisfies
\[
	\begin{bvp}
	(-\Delta)^s u_\Omega &= 0, &&\text{in }\R^n\setminus \Omega, \\
	u_\Omega &=1, &&\text{in }\Omega.
	\end{bvp}
\]

\subsection{The extended formulation}\label{sec:extd}

The proof of our main result strongly relies on an extension procedure for functions in fractional Sobolev spaces, first established in \cite{Caffarelli2007}, which, in some sense, allows to avoid some nonlocal issues and recover a local framework. Such a procedure is commonly called \emph{Caffarelli-Silvestre extension}. In this paragraph, we introduce the functional spaces emerging in the extended formulation and we discuss some of their properties, also in relation with the $s$-capacity of a set. We remark that, being the Caffarelli-Silvestre extension a classical tool nowadays, the results we present here can be regarded as folklore. But still, up to our knowledge, there are no explicit proofs available in the literature, hence we decided to report them here.

For any closed set $K\sub\partial\R^{n+1}_+\simeq \R^n$, we define the space $\mathcal{D}^{1,2}(\overline{\R^{n+1}_+}\setminus K;z^{1-2s})$ as the completion of $C_c^\infty(\overline{\R^{n+1}_+}\setminus K)$ with respect to the norm 
\[
	\norm{U}_{\mathcal{D}^{1,2}(\overline{\R^{n+1}_+}\setminus K;z^{1-2s})}:=\left(\int_{\R^{n+1}_+}z^{1-2s}\abs{\nabla U}^2\,dx\,dz\right)^{\frac{1}{2}}.
\]
However  hereafter we simply write 
	$\mathcal{D}^{1,2}_z(\overline{\R^{n+1}_+}\setminus K)$ in place of $\mathcal{D}^{1,2}(\overline{\R^{n+1}_+}\setminus K;z^{1-2s})$.

We have that $\mathcal{D}^{1,2}_z(\overline{\R^{n+1}_+}\setminus K)$ is an Hilbert space with respect to the scalar product
\[
(U,V)_{\mathcal{D}^{1,2}_z(\overline{\R^{n+1}_+}\setminus K)}:=\int_{\R^{n+1}_+}z^{1-2s}\nabla U\cdot\nabla V\,dx\,dz.
\]
First of all we shot that the space $\mathcal{D}^{1,2}_z(\overline{\R^{n+1}_+}\setminus K)$ is a  well defined functional space, which is not obvious. In order to prove that, it is sufficient to prove that it is the case for $\mathcal{D}^{1,2}_z(\overline{\R^{n+1}_+})$, in view of the continuous embedding
\[
\mathcal{D}^{1,2}_z(\overline{\R^{n+1}_+}\setminus K)\hookrightarrow \mathcal{D}^{1,2}_z(\overline{\R^{n+1}_+}).
\]
To show this, we first recall by \cite[Proposition 3.3]{DMV2017} the following weighted Sobolev inequality
\begin{equation}\label{sobolev}
\left(\int_{\R^{n+1}_+} z^{1-2s}\abs{U}^{2\gamma}\,dx\,dz  \right)^\frac{1}{2\gamma}\le S_{n,s}'\left( \int_{\R^{n+1}_+} z^{1-2s}|\nabla U|^{2}\,dx\,dz\right)^{\frac12}\quad\text{for all }U\in \Dext,
\end{equation}
where $S_{n,s}'$ is a positive constant and $\gamma:=1+\frac{2}{n-2s}$. 
In particular this inequality yields the following continuous embedding
\[
	\mathcal{D}^{1,2}_z(\overline{\R^{n+1}_+})\hookrightarrow L^{2\gamma}(\R^{n+1}_+;z^{1-2s}),
\]
where 
\[
	L^{2\gamma}(\R^{n+1}_+;z^{1-2s}):=\left\{ U\in L^1_{\textup{loc}}(\R^{n+1}_+)\colon \int_{\R^{n+1}_+}z^{1-2s}\abs{U}^{2\gamma}\,dx\,dz <\infty \right\}.
\]
We now provide a characterization of $\Dext$ as a concrete functional space.
\begin{proposition}\label{D-Sobolev}
The space $\mathcal{D}^{1,2}_z(\overline{\R^{n+1}_{+}})$ is a functional space. In particular there holds
\[
\mathcal{D}^{1,2}_z(\overline{\R^{n+1}_+})=\left\{ U\in L^{2\gamma}(\R^{n+1}_+;z^{1-2s})\colon  \norm{U}_{\mathcal{D}^{1,2}_z(\overline{\R^{n+1}_+})}<+\infty  \right\}.
\]
\end{proposition}
\begin{proof}
The fact that
\[
	\mathcal{D}^{1,2}_z(\overline{\R^{n+1}_+})\sub\left\{ U\in L^{2\gamma}(\R^{n+1}_+;z^{1-2s})\colon  \norm{U}_{\mathcal{D}^{1,2}_z(\overline{\R^{n+1}_+})}<+\infty  \right\}
\]
immediately follows from \eqref{sobolev}. We now prove the reverse inclusion. Namely, we show that any function
\begin{equation}\label{eq:char_1}
 	U\in L^{2\gamma}(\R^{n+1}_+;z^{1-2s})
\end{equation}
such that
\begin{equation}\label{eq:char_2}
	\norm{U}_{\mathcal{D}^{1,2}_z(\overline{\R^{n+1}_+})}<+\infty
\end{equation}
can be approximated by functions in $C_c^\infty(\overline{\R^{n+1}_+})$ in the topology induced by the norm $\norm{\cdot}_{\Dext}$. First, suppose that $U$ is compactly supported in $\overline{\R^{n+1}_+}$ and let
\[
	\tilde{U}(x,z):=\begin{cases}
		U(x,z),&\text{if }z>0, \\
		U(x,-z), &\text{if }z<0.
	\end{cases}
\]
 Moreover, we let $\{\rho_{\eps}\}_{\eps>0}$ be a family of mollifiers\footnote{We call  mollifier a smooth, symmetric decreasing, positive and compactly supported function, which converges in distribution to a centered    Dirac measure $\delta$, as $\eps\to 0$, and such that $\|\rho_{\eps}\|_{L^1(\Rnn)}=1$.} in $\R^{n+1}$ and we set 
 \[
 	U_{\eps}={\tilde{U}\star\rho_{\eps}}_{\,\big|\overline{\R^{n+1}_+}}.
 \]
Clearly $U_{\eps}$ pointwisely converge to $U$ in $\overline{\R^{n+1}}$, as $\eps\to 0$. Moreover it is equibounded in $\Dext$. Thus we easily conclude by means of the dominated convergence theorem.
We consider now  the general case. Fix $\eps>0$ and let $U_{R}=U\eta_{R}$ where $\eta_{R}$ is the restriction to $\overline{\R^{n+1}_+}$ of a radial, smooth cut-off function defined on $\R^{n+1}$ such that $\eta_{R}=1$ on $B_{R}$, $\eta_{R}=0$ on $\R^{n+1}\setminus B_{2R}$ and $\sup_{\R^{n+1}}|\nabla\eta_{R}|\le 4R^{-1}$. Since, $U_{R}\in \Dext$, by the previous step there exists $V_{R}\in C^{\infty}_{c}(\overline{\R^{n+1}_{+}})$ such that $\norm{U_{R}-V_{R}}_{\Dext}\le \eps/2$, so that, by triangular inequality
\[
\norm{U-V_{R}}_{\Dext} \le \norm{U-U_{R}}_{\Dext}+\frac{\eps}{2}.
\]
We are left to show that $U_{R}\to U$ in $\Dext$, as $R\to\infty$. In view of the properties of $\eta_R$, we have that
\[
\begin{aligned}
\norm{U-U_{R}}^{2}_{\Dext}&=\int_{\Rnn_+}z^{1-2s}|\nabla U-\nabla (U\eta_{R})|^{2}\,dx\,dz\\
&\le 2\int_{\Rnn_+}z^{1-2s}|\nabla U-\eta_{R}\nabla U|^{2}\,dx\,dz+2\int_{\Rnn_+}z^{1-2s}|U\nabla\eta_{R}|^{2}\,dx\,dz\\
&\le 2\int_{\Rnn_+\setminus B_{R}^+}z^{1-2s}|\nabla U|^{2}\,dx\,dz+\frac{32}{R^2}\int_{B_{2R}^+\setminus B_{R}^+}z^{1-2s}\abs{U}^{2}\,dx\,dz,
\end{aligned}
\]
where $B_r^+:=B_r\cap\Rnn_+$. Thanks to \eqref{eq:char_2}, the first term on the right-hand side in the above inequalities is infinitesimal as $R$ tends to infinity, so it can be chosen smaller than $\eps/4$. For what concerns the second term, by  H\"older inequality we obtain that
\[
\int_{B_{2R}^+\setminus B_{R}^+}z^{1-2s}\abs{U}^{2}\,dx\,dz\le \left(\int_{B_{2R}^+\setminus B_{R}^+} z^{1-2s}\,dx\,dz \right)^{(\gamma-1)/\gamma} \left(\int_{B_{2R}^+\setminus B_{R}^+}z^{1-2s}\abs{U}^{2\gamma}\,dx\,dz\right)^{1/\gamma}.
\]
By \eqref{eq:char_1} we have that
\[
	\int_{B_{2R}^+\setminus B_{R}^+}z^{1-2s}\abs{U}^{2\gamma}\,dx\,dz\to 0,\quad\text{as }R\to \infty,
\]
while, by an explicit computation, also recalling that $\gamma=1+\frac{n}{n-2s}$ one gets that
\[
\sup_{R\ge1}\frac{1}{R^2} \left(\int_{B_{2R}^+\setminus B_{R}^+} z^{1-2s}\,dx\,dz \right)^{(\gamma-1)/\gamma}<+\infty.
\]
Hence we can choose $R$ large enough so that
\[
\frac{16}{R^2}\int_{B_{2R}^+\setminus B_{R}^+}z^{1-2s}U^{2}\,dx\,dz\le\eps/4,
\] 
and conclude that
\[
\norm{U-V_{R}}_{\Dext}\leq \eps.
\]
\end{proof}

Another fundamental fact that relates the space $\Dext$ with $\Ds$ is the existence of a trace map from the former to the latter. Before stating the precise result, we recall the following classical Hardy inequality, whose proof can be found e.g. in \cite{HLP1952}. 
\begin{lemma}\label{lemma:hardy}
	Let $p\in(1,\infty)$ and $a<1$. Then there exists a constant $C(p,a)>0$ such that
	\[
		\int_0^\infty \rho^a\abs{\frac{1}{\rho}\int_0^\rho f(t)\,dt}^p\,d\rho\leq C(p,a)\int_0^\infty\rho^a\abs{f(\rho)}^p\,d\rho,
	\]
	for all $f\in C_c^\infty([0,\infty))$.
\end{lemma}

\begin{proposition}\label{prop:trace}
	There exists a linear and continuous trace operator
	\[
		\Tr \colon \Dext\to \Ds
	\]
	such that $\Tr(U)(x)=U(x,0)$ for every $U\in C^{\infty}_{c}(\overline{ \R^{n+1}_{+}})$.
\end{proposition}
\begin{proof}
	
	Throughout the proof, we assume the space $\Ds$ to be endowed with the following norm
	\[
		[u]_{s,\#}:=\left( \sum_{i=1}^n[u]_{s,i}^2 \right)^{\frac{1}{2}}.
	\]
	where
	\[
		[u]_{s,i}^2:=\int_0^\infty\int_{\R^n}\frac{\abs{u(x+\rho \bm{e}_i)-u(x)}^2}{\rho^{1+2s}}\,dx\,d\rho,
	\]
	with $\bm{e}_i$ denoting the unit vector in the positive $x_i$ variable. The norm $[\cdot]_{s,\#}$ is equivalent to $[\cdot]_s$, as proved in \cite[Proposition B.1]{BCV}. By density, it is sufficient to prove that there exists $C>0$ such that
	\begin{equation}\label{eq:trace_1}
		[U(\cdot,0)]_{s,\#}\leq C\norm{U}_{\Dext},\quad\text{for all }U\in C_c^\infty(\overline{\R^{n+1}_+}).
	\end{equation}
	For $U\in C_c^\infty(\overline{\R^{n+1}_+})$, $x\in\R^n$ and $\rho>0$, we rewrite
	\begin{gather*}
		U(x,0)=U(x,\rho)-\int_0^\rho\frac{\partial U}{\partial z}(x,t)\,dt, \\
		U(x+\rho \bm{e}_i,0)=U(x+\rho \bm{e}_i,\rho)-\int_0^\rho\frac{\partial U}{\partial z}(x+\rho \bm{e}_i,t)\,dt.
	\end{gather*}
	Therefore
	\begin{multline*}
		\frac{\abs{U(x+\rho\bm{e}_i,0)-U(x,0)}^2}{\rho^{1+2s}}\leq 2\rho^{1-2s}\frac{\abs{U(x+\rho\bm{e}_i,\rho)-U(x,\rho)}^2}{\rho^2}\\
		+ 2\rho^{1-2s}\abs{\frac{1}{\rho}\int_0^\rho\left( \frac{\partial U}{\partial z}(x+\rho\bm{e}_i,t)-\frac{\partial U}{\partial z}(x,t) \right)\,dt }^2.
	\end{multline*}
	If we integrate in the $x$ variable we obtain 
	\begin{multline*}
		\int_{\R^n}\frac{\abs{U(x+\rho\bm{e}_i,0)-U(x,0)}^2}{\rho^{1+2s}}\,dx\leq 2\rho^{1-2s}\int_{\R^n}\abs{\frac{\partial U}{\partial x_i}(x,\rho)}^2\,dx\\
		+4\rho^{1-2s}\int_{\R^n}\abs{\frac{1}{\rho}\int_0^\rho\left(\frac{\partial U}{\partial z}(x,t)\,dt \right) }^2\,dx,
	\end{multline*}
	where we used the fact that
	\[
		\int_{\R^n}\frac{\abs{U(x+\rho\bm{e}_i,\rho)-U(x,\rho)}^2}{\rho^2}\,dx \leq \int_{\R^n}\abs{\frac{\partial U}{\partial x_i}(x,\rho)}^2\,dx
	\]
	for the first term and a change of variable for the second. By integration with respect to $\rho$ in $(0,\infty)$ and thanks to Lemma \ref{lemma:hardy} (choosing $p=2$ and $a=1-2s$) we infer
	\[
		[U(\cdot,0)]_{s,i}^2\leq 2\int_{\R^{n+1}_+}\rho^{1-2s}\abs{\frac{\partial U}{\partial x_i}(x,\rho)}^2\,dx\,d\rho +C\int_{\R^{n+1}_+}\rho^{1-2s}\abs{\frac{\partial U}{\partial z}(x,\rho)}^2\,dx\,d\rho,
	\]
	for some constant $C>0$ depending only on $s$. If we now sum for $i=1,\dots,n$ and take the square root, we obtain \eqref{eq:trace_1}, thus concluding the proof.

\end{proof}

The next step is to prove that the trace map introduced in the proposition above is onto. We first introduce the Poisson kernel of the upper half-space $\R^{n+1}_+$, defined as
\begin{equation*}
	P_z(x):=c_{n,s}\frac{z^{2s}}{(\abs{x}^2+z^2)^{\frac{n+2s}{2}}},\quad\text{for }(x,z)\in\R^{n+1}_+
\end{equation*}
where
\begin{equation}\label{eq:c}
	c_{n,s}:=\left( \int_{\R^n}\frac{1}{(1+\abs{x}^2)^{\frac{n+2s}{2}}}\,dx \right)^{-1}=\pi^{-\frac{n}{2}}\frac{\Gamma\left(\frac{n+2s}{2}\right)}{\Gamma(s)},
\end{equation}
is given in such a way that
\begin{equation*}
	\int_{\R^n}P_z(x)\,dx=1,\quad\text{for all }z>0,
\end{equation*}
see \cite[Remark 2.2]{FMM}. Essentially, a convolution with this kernel allows to extend to the upper half-space $\Rnn_+$ functions that are defined on $\R^n$. This is done first for functions in $C_c^\infty(\R^n)$ and then extended by density to the whole $\Ds$. Namely, we have the following.

\begin{proposition}\label{ext}
	Let $\varphi\in\Ds$. Then the function
	\begin{equation}\label{eq:ext_poisson_th1}
		U_\varphi(x,z):=(P_z\star \varphi)(x)
	\end{equation}
	belongs to $\Dext$ and 
	\begin{equation}\label{eq:ext_poisson_th2}
		\norm{U_\varphi}_{\Dext}^2=\alpha_{n,s}[\varphi]_s^2,
	\end{equation}
	where
	\begin{equation}\label{eq:alpha}
		\alpha_{n,s}:=\frac{s(1-s)}{\pi^{\frac{n}{2}}}\frac{\Gamma(1-s)\Gamma\left(\frac{n+2s}{2}\right)}{\Gamma(s)\Gamma(2-s)}>0.
	\end{equation}
	In particular the trace operator established in Proposition \ref{prop:trace} is surjective.
\end{proposition}
\begin{proof}
	For any $\varphi\in C_c^\infty(\R^n)$, we let
	\[
		(\mathrm{E}\varphi)(x,z):=(P_z\star \varphi)(x).
	\]
	It is easy to check that
	\begin{equation}\label{eq:ext_poisson_1}
		\norm{\mathrm{E}\varphi}_{\Dext}^2=\int_{\R^{n+1}_+}z^{1-2s}\abs{\nabla (\mathrm{E}\varphi)}^2\,dx\,dz=\alpha_{n,s}[\varphi]_s^2.
	\end{equation}
	For the computation of the explicit constant see e.g. \cite[Remark 3.11]{Cabre2014}. {Moreover, by the weighted Sobolev inequality \eqref{sobolev}
	\[
		\int_{\R^{n+1}_+}z^{1-2s}\abs{\mathrm{E}\varphi}^{2\gamma}\,dx\,dz<\infty.
	\]
	Hence, by Proposition \ref{D-Sobolev}, we have that $\mathrm{E}\varphi\in \Dext$.} Therefore the map
	\[
		\mathrm{E}\colon C_c^\infty(\R^n)\to \Dext
	\]
	is linear and continuous, thus it can be uniquely extended in the whole $\Ds$ and \eqref{eq:ext_poisson_1} still holds. We now prove that $\mathrm{E}\varphi=U_\varphi$ for any $\varphi\in\Ds$. Let $\varphi\in\Ds$ and $(\varphi_i)_i\sub C_c^\infty(\R^n)$ be such that $\varphi_i\to\varphi$ in $\Ds$ as $i\to\infty$. Thanks to \eqref{eq:sobolev_ineq}, we have that
	\[
		\varphi_i\to \varphi\quad\text{in }L^{2^*_s}(\R^n),\quad\text{as }i\to\infty,
	\]
	which, by definition, implies that
	\[
		\mathrm{E}\varphi_i=U_{\varphi_i}\to U_{\varphi}\quad\text{pointwise in }\R^{n+1}_+,\quad\text{as }i\to\infty.
	\]
	On the other hand $\mathrm{E}\varphi_i\to \mathrm{E}\varphi$ in $\Dext$ as $i\to\infty$; hence, up to a subsequence
	\[
		\mathrm{E}\varphi_i\to \mathrm{E}\varphi\quad\text{a.e. in }\R^{n+1}_+,\quad\text{as }i\to\infty.
	\]
	Therefore $\mathrm{E}\varphi=U_{\varphi}$ and the proof is complete.

\end{proof}

The following corollary, in view of the previous results, tells us that there is an isometry between the space $\Ds$ and the subspace of $\Dext$ containing the (unique) minimizers of a certain functional.

\begin{corollary}\label{cor-ext}
	Let $\varphi\in\Ds$. Then the minimization problem
	\begin{equation*}
		\min_{U\in\Dext}\left\{ \int_{\R^{n+1}_+}z^{1-2s}\abs{\nabla U}^2\,dx\,dz\colon \Tr U=\varphi \right\}
	\end{equation*}
	admits a unique solution, which coincides with $U_\varphi$ as in \eqref{eq:ext_poisson_th1}.	Moreover
	\begin{equation*}
		\begin{bvp}
		-\dive(z^{1-2s}\nabla U_\varphi)&=0, &&\text{in }\R^{n+1}_+, \\
		U_\varphi&=\varphi, &&\text{on }\R^n, \\
		-\lim_{z\to 0^+}z^{1-2s}\frac{\partial U_\varphi}{\partial z}&=\alpha_{n,s} (-\Delta)^s\varphi,&&\text{on }\R^n,
		\end{bvp}
	\end{equation*}
	in a weak sense, that is
	\[
	\int_{\R^{n+1}_+}z^{1-2s}\nabla U_\varphi\cdot\nabla V\,dx\,dz=\alpha_{n,s}(\varphi,\Tr V)_{\Ds},\quad\text{for all }V\in\Dext,
	\]
	where $\alpha_{n,s}>0$ is as in \eqref{eq:alpha}.
\end{corollary}

\begin{proof}
	The proof is standard. For instance see \cite[Proposition 2.6]{BCV} for the first part and \cite{Caffarelli2007} for the second.
\end{proof}

\begin{remark}\label{rmk:ext_cap}
	In view of the extension procedure described above, we can relate the notion of $s$-capacity with another notion of (weighted) capacity of sets in $\R^{n+1}$. More precisely, with a slight abuse of notation, we call $\Dzc$ the completion of $C_c^\infty(\Rnn)$ with respect to the norm
\[
	\norm{U}_{\Dzc}:=\left(\int_{\Rnn}\abs{z}^{1-2s}\abs{\nabla U}^2\,dx\,dz\right)^{\frac{1}{2}}
\]
and, for any closed $K\sub\R^{n+1}$, we let
\begin{equation*}
	\C_{\R^{n+1}}(K;\abs{z}^{1-2s}):=\inf\left\{ \norm{U}_{\mathcal{D}^{1,2}_z(\R^{n+1})}^2\colon U\in\mathcal{D}^{1,2}_z(\R^{n+1}),~ U-\zeta_K\in\mathcal{D}^{1,2}(\R^{n+1}\setminus K) \right\},
\end{equation*}
where $\zeta_K\in C_c^\infty(\R^{n+1})$ is equal to $1$ in a neighborhood of $K$. Thanks to \eqref{eq:ext_poisson_th2}, after an even reflection in the $z$ variable, one can see that
\[
	\frac{1}{2}\C_{\R^{n+1}}(\Omega;\abs{z}^{1-2s})=\alpha_{n,s}\Cs(\Omega).
\] 
for any closed $\Omega\sub\R^n$. Moreover the unique function achieving $\C_{\R^{n+1}}(\Omega;\abs{z}^{1-2s})$ coincides with the even-in-$z$ reflection of $P_z\star u_\Omega$, with $u_\Omega\in\Ds$ denoting the capacitary potential of $\Cs(\Omega)$, as in Remark \ref{rmk:potential}. 

Hereafter we denote by
\begin{equation}\label{eq:potential_ext}
	U_\Omega(x,z):=(P_z\star u_\Omega)(x),\quad \text{for }(x,z)\in\R^{n+1}_+
\end{equation}
the restriction to the upper half-space of the potential associated to  $\C(\Omega;\abs{z}^{1-2s})$. Hence, $U_\Omega\in\Dext$ satisfies $0\leq U_\Omega\leq 1$ a.e. in $\R^{n+1}_+$ and weakly solves
\[
	\begin{bvp}
	-\dive(z^{1-2s}\nabla U_\Omega)&=0, &&\text{in }\R^{n+1}_+, \\
	U_\Omega&=1, &&\text{on }\Omega, \\
	-\lim_{z\to 0^+} z^{1-2s}\frac{\partial U_\Omega}{\partial z}&=0,&&\text{on }\R^n\setminus\Omega,
	\end{bvp}
\]
in the sense that $U_\Omega-\zeta_\Omega\in\DextO$ and
\[
	\int_{\R^{n+1}_+}z^{1-2s}\nabla U_\Omega\cdot\nabla V\,dx\,dz=0,\quad\text{for all }V\in\DextO.
\]
Thanks to \cite[Theorem 1.1]{Sire2021}, we can observe that $U_\Omega\in C^\infty(\overline{ \R^{n+1}_{+}}\setminus\Omega)$. Hence, in particular $u_\Omega\in C^\infty(\R^n\setminus\Omega)$.
\end{remark}

\begin{remark}\label{rmk:functional_setting}
	We emphasize a technical difference with respect to \cite{BCV}. Namely  that the functional setting we adopt here is tailored for the problem under investigation. Indeed, being the $s$-capacity obtained by minimization of the (nonlocal) energy, it is natural to expect the minimizer to have only finite seminorm $[\cdot]_s$, together with the possibility of approximating it by means of smooth and compactly supported functions, at least for compact $\Omega$. Therefore, any  other integrability assumption on the capacitary potential appears as artificial. Observe that, differently from our Proposition \ref{ext} and Corollary \ref{cor-ext}, in which the trace function $\varphi$ needs just to belong to $\mathcal D^{s,2}(\R^n)$, in the analogue extension result Proposition 2.6 of \cite{BCV} an additional integrability assumption on $\varphi$ is required. 
\end{remark}

\subsection{ Radial rearrangements and the isocapacitary inequality}


This paragraph is devoted to the isocapacitary inequality \eqref{eq:frac_iso}. The classical and simplest proof of the (standard) isocapacitary inequality,
\begin{equation}\label{isocap}
	|\Omega|^{(2-n)/n}\C(\Omega)\geq  |B|^{(2-n)/n}\C(B)
\end{equation}
is by rearrangement: given any function $u:\R^n\to\R$, its symmetric decreasing rearrangement is the radial decreasing function $u^*:\R^n\to\R$ such that $|\{u>t\}|=|\{u^*>t\}|$. As a consequence of the P\'olya-Szeg\"o inequality
\[
\int_{\R^n} |\nabla u|^2\, dx\ge \int_{\R^n} |\nabla u^*|^2\, dx
\]
applied to the capacitary potential of a closed $\Omega\sub\R^n$ one can easily see that $\C(\Omega)\ge \C(B)$ as long as $|B|=|\Omega|<\infty$, from which, by scaling \eqref{isocap} holds as well. Indeed, the symmetric rearrangement of the capacitary potential of $\Omega$ coincides with the potential of a ball with the same volume as $\Omega$. Following this path, one can prove the fractional isocapacitary inequality using symmetric rearrangements for the extended problem in $\R^{n+1}_+$.

As in \cite{BCV,FMM}, we define in $\R^{n+1}_+$ the {\it partial Schwartz symmetrization} $U^*$ of a nonnegative function $U\in \mathcal{D}_z^{1,2}(\R^{n+1}_+)$. By construction, the function $U^*$ is obtained by taking for almost every $z>0$, the $n-$dimensional Schwartz symmetrization of the map
\[
x\mapsto U(x,z).
\]
More precisely: for almost every fixed $z>0$, the function $U^*(\cdot, z)$ is defined to be the unique radially symmetric decreasing function on $\mathbb R^n$ such that for all $t>0$
$$
|\{ U^*(\cdot, z)>t\}|=|\{  U(\cdot, z)>t\}|.
$$
\begin{proposition}\label{PS}
Let $\varphi\in \mathcal{D}^{s,2}(\R^{n})$ be a nonnegative function and let $U_\varphi\in\mathcal{D}^{1,2}_z(\overline{\R^{n+1}_+})$ as in \eqref{eq:ext_poisson_th1}. Then 
\[
U^*_\varphi\in \mathcal{D}^{1,2}_z(\overline{\R^{n+1}_+})
\]
and the following P\'olya-Szeg\"{o} type inequalities hold true
\begin{gather}
\int_{\mathbb{R}^{n+1}_+} z^{1-2\,s}\,|\nabla U_\varphi^*|^2\,dx\,dz\leq\int_{\mathbb{R}^{n+1}_+} z^{1-2\,s}\,|\nabla U_\varphi|^2\,dx\,dz, \label{polyaszego}\\
\int_{\mathbb{R}^{n+1}_+} z^{1-2\,s}\,|\partial_{z}U_\varphi^*|^2\,dx\,dz\leq\int_{\mathbb{R}^{n+1}_+} z^{1-2\,s}\,|\partial_z U_\varphi|^2\,dx\,dz.\label{polyaszegoz}
\end{gather}
Moreover, we have $\mathrm{Tr}(U_\varphi^*)=\varphi^*$.
In particular, we get
\begin{equation}
\label{energia}
\int_{\mathbb{R}^{n+1}_+} z^{1-2\,s}\,|\nabla U_\varphi^*|^2\,dx\,dz\ge \alpha_{n,s} [\varphi^*]_{\mathcal{D}^{s,2}(\R^{n})}^2.
\end{equation}
\end{proposition}

\begin{proof}
By density, it is enough to show the result for $\varphi\in C^{\infty}_{c}(\R^{n})$. In that case, the proof follows directly by \cite[Proposition 3.2]{BCV}.
\end{proof}

{Now, a proof of the fractional isocapacitary inequality \eqref{eq:frac_iso} can be obtained as a direct consequence the previous result, applied to $\varphi=u_\Omega$, and Remark \ref{rmk:ext_cap}.  }

\section{Proof of the main result}\label{sec:proof}
In this Section we give the proof of our main result. The idea consists in introducing quantitative elements in the proof of the isocapacitary inequality established in the previous section.
As already explained in the Introduction, the major inconvenience when working with the extended problem in $\R^{n+1}_+$ consists in the fact that we need to transfer information on the superlevel sets of the extension of the capacitary potential $\{ U_\Omega(\cdot,z)\ge t\}$ for fixed $z>0$, to information on the superlevel sets of its trace $u_\Omega$ in $\R^n$. This was done in Section 4 of \cite{BCV} for a problem concerning the stability of the first eigenvalue of the Dirichlet fractional Laplacian.

We start by recalling the following technical result, whose proof can be found in \cite[Lemma 4.1]{BCV}.
\begin{lemma}\label{trasfasimmetria}
Let $\Omega,E$ be two measurable subsets of $\R^n$ of finite measure and such that
\[
\frac{|\Omega\Delta E|}{|\Omega|}\le\frac{\delta}{3}\A(\Omega),
\]
for some $\delta\in(0,1)$. Then
\[
\A(E)\ge C_\delta \A(\Omega),
\]
{where
\begin{equation}\label{eq:c_gamma}
	C_\delta:=\frac{3-2\delta}{3+2\delta}.
\end{equation}}
\end{lemma}

{The following lemma corresponds to Proposition 2.6 of \cite{BCV}. Observe that, 
here, the extension {$U_\varphi(\cdot, z)$ of a function $\varphi\in\Ds$} does not belong to $L^2(\mathbb R^n)$ for fixed $z\geq 0$; nevertheless, with the same computations as in \cite{BCV}, we can show that it is close enough (depending on $z$) in $L^2$ to its trace $u$. We report a proof of the lemma for the sake of completeness.}
\begin{lemma}\label{L2}
	{For any $\varphi\in \Ds$, denoting by $U_\varphi\in\Dext$ its extension,} there holds
	\begin{equation*}
		\norm{U_\varphi(\cdot,z)-\varphi}_{L^2(\R^n)}\leq {\sqrt{{c}_{n,s}}} [\varphi]_s\,z^{s},\quad\text{for  }z>0,
	\end{equation*}
	
	{where $c_{n,s}$ is given in \eqref{eq:c}}.
\end{lemma}
\begin{proof}
	We first prove the following preliminary fact
	\begin{equation}\label{eq:prelim}
		\int_{\R^n}P_z(y)\norm{\tau_y \varphi -\varphi }_{L^2(\R^n)}\,dy\leq \sqrt{c_{n,s}}z^s[\varphi]_s<\infty\quad\text{for all }\varphi\in \Ds~\text{and all }z\geq 0.
	\end{equation}
	Indeed, multiplying and dividing by $\abs{y}^{(n+2s)/2}$ and applying Cauchy-Schwartz inequality yields
	\begin{align*}
		\int_{\R^n}P_z(y)\norm{\tau_y \varphi -\varphi }_{L^2(\R^n)}\,dy &\leq\left(\int_{\R^n} P_z(y)^{2}\abs{y}^{n+2s}\,dy \right)^{\frac{1}{2}}\left( \int_{\R^n}\norm{\frac{\tau_y \varphi -\varphi }{\abs{y}^s}}_{L^2(\R^n)}^2\frac{dy}{\abs{y}^n} \right)^{\frac{1}{2}} \\
		&=\left(\int_{\R^n} P_z(y)^{2}\abs{y}^{{n+2s}}\,dy \right)^{\frac{1}{2}}[\varphi ]_s,
	\end{align*}
	where, in the last step, we used the fact that
	\[
	\int_{\R^n}\norm{\frac{\tau_y \varphi -\varphi }{\abs{y}^s}}_{L^2(\R^n)}^2\frac{dy}{\abs{y}^n}=[\varphi ]_s^2.
	\]
	The proof of \eqref{eq:prelim} ends by observing that
	\[
	\left(\int_{\R^n} P_z(y)^{2}\abs{y}^{n+2s}\,dy \right)^{\frac{1}{2}}={\sqrt{{c}_{n,s}}} z^{s}.
	\]
	Let us now consider $\varphi\in\Ds$. By definition of $U_\varphi $ and Minkowski's inequality we have that
	\begin{equation*}
		\norm{(U_\varphi (\cdot,z)-\varphi)\chi_{B_R} }_{L^2(\R^n)}\leq \int_{\R^n}P_z(y)\norm{\tau_y \varphi -\varphi }_{L^2(B_R)}\,dy,
	\end{equation*}
	for any $R>0$, where $\chi_{B_R}$ denotes the characteristic function of the ball $B_R$. We conclude the proof applying Fatou's Lemma for $R\to\infty$ and combining the resulting inequality with \eqref{eq:prelim}.
\end{proof}

{The following result allows us to focus on compact sets without loss of generality. Therefore hereafter in this section we always assume $\Omega\sub\R^n$ to be a compact set.}

\begin{lemma}[Reduction to compact sets]\label{lemma:compact}
	Let $\Omega\sub\R^n$ be a closed set with finite measure. Then
	\[
	\lim_{r\to\infty}\Cs(\Omega\cap \overline{B_r})=\Cs(\Omega)\quad\text{and}\quad \lim_{r\to\infty}\mathcal{A}(\Omega\cap \overline{B_r})=\mathcal{A}(\Omega),
	\]
	where $B_r:=\{x\in\R^n\colon \abs{x}<r \}$.
\end{lemma}
\begin{proof}
	The convergence of the capacity follows from the fact that, for any sequence of closed sets $\Omega_i\sub\R^n$ such that $\Omega_i\sub\Omega_{i+1}$, there holds
	\[
	\Cs(\cup_{i=1}^\infty \Omega_i)=\lim_{i\to\infty}\Cs(\Omega_i)
	\]
	which, in turn, can be proved by following step by step the proof of \cite[Theorem 4.15, Point \textrm{(viii)}]{EG2015}. The rest of the proof is easy and can be omitted.
\end{proof}

{Hereafter in this section, for $0\leq t\leq 1$ and $z\geq 0$, we let
\begin{equation*}
	\Omega_{t,z}:=\{ x\in\R^n\colon U_\Omega(x,z)\geq t \}\quad\text{and}\quad \Omega_t:=\Omega_{t,0}=\{x\in\R^n\colon u_\Omega(x)\geq t\}.
\end{equation*}
We notice that, by the continuity of $u_{\Omega}$, it follows that $|\Omega_t\Delta\Omega|\rightarrow 0$, as $t\to 1^{-}$. Moreover, we set, respectively
\begin{equation*}
	\mu_z(t):=|\Omega_{t,z}|\quad\text{and}\quad \mu(t):=\mu_0(t)=|\Omega_t|.
\end{equation*}
We immediately observe that $\mu$ is left-continuous and non-increasing in $(0,1)$ and that
\[
	\lim_{t\to 0^+}\mu(t)=+\infty\quad\text{and}\quad \lim_{t\to 1^-}\mu(t)=|\Omega|,
\]
where the last equality is a consequence of the weak maximum principle. We now let}
\begin{align}
	T=T(\Omega,{\gamma})&=\inf\left\{0\leq t\leq 1\colon|\{u_\Omega\geq t\}| \le |\Omega|\left(1+\gamma\mathcal A(\Omega)\right)\right\}, \label{T}\\
	&=\inf\left\{0\leq t\leq 1\colon \mu(t) \le |\Omega|\left(1+\gamma\mathcal A(\Omega)\right)\right\},\notag
\end{align}
The constant $\gamma$ is chosen in {$(0,1/9)$} and will be settled later on.
Notice that if $\A(\Omega)>0$ then $T<1$. {In addition, in view of the left-continuity of $\mu$, we know that
\begin{equation}\label{eq:mu_T}
	\mu(T)\geq\abs{\Omega}(1+\gamma\mathcal{A}(\Omega)).
\end{equation}}

The following proposition, which will be crucial in the proof of our main result, allows to bound from below the asymmetry of the superlevel sets of $U_\Omega(\cdot, z)$ with the asymmetry of $\Omega$ (for certain levels $t$ and for $z$ small enough).

\begin{proposition}
\label{prop:asimmetrie}
Let $\Omega$ be such that $\A(\Omega)>0$, {$\gamma \in (0, 1/9)$}, and let  $T\in(0,1)$ as in \eqref{T}. Set also $\widehat T=1-T$. Then, letting ${c}_{n,s}$ be as in \eqref{eq:c}, if
\[
 T+\frac18 \widehat T \le t\le T+\frac{3}{8}\widehat T \quad \mbox{and}\quad0<z\le z_{0}:={\left(\frac{\widehat{T}}{16}\sqrt{\frac{\gamma\mathcal{A}(\Omega) |\Omega|}{ {c}_{n,s}{\Cs(\Omega)}}}\right)^\frac{1}{s}},
\]
we have
\begin{equation}
\label{ok}
\frac{\Big|\big|\Omega_{t,z}\big|-|\Omega|\Big|}{|\Omega|}\le 3\gamma\,\mathcal{A}(\Omega),
\end{equation}
and
\begin{equation}
\label{asimmetria}
\mathcal{A}\big(\Omega_{t,z}\big)\ge c_\gamma\mathcal{A}(\Omega),
\end{equation}
where $c_\gamma=C_{3\gamma}$ and $C_{3\gamma}$ is as in \eqref{eq:c_gamma}.
\end{proposition}
\begin{proof}
Let $\tau=T+\frac{\widehat T}{2}$. First of all, we observe that, {by triangle inequality}
\begin{equation}\label{smiths}
\begin{aligned}
{\frac{||\Omega_{t,z}|-|\Omega||}{|\Omega|}\leq}\frac{|\Omega_{t,z}\Delta\Omega|}{|\Omega|}&= \frac{|\Omega_{t,z}\setminus\Omega|}{|\Omega|}+\frac{|\Omega\setminus \Omega_{t,z}|}{|\Omega|}\\
& \le\frac{|\Omega_{t,z}\setminus\Omega_{T+\frac{\widehat{T}}{16}}|}{|\Omega|} + \frac{|\Omega_{T+\frac{\widehat T}{16}}\setminus \Omega|}{|\Omega|} +  \frac{|\Omega\setminus \Omega_{t,z}|}{|\Omega|}\\
&\le \frac{|\Omega_{t,z}\setminus\Omega_{T+\frac{\widehat{T}}{16}}|}{|\Omega|}+\gamma\A(\Omega)+ \frac{|\Omega_{\tau}\setminus \Omega_{t,z}|}{|\Omega|}.
\end{aligned}
\end{equation}
where, in the last inequality we used {the definition of $T$ and the facts} that $T+\frac{\widehat{T}}{16}>T$ and  $\Omega\subseteq\Omega_\tau$. For $x\in\Omega_{\tau}\setminus \Omega_{t,z}$ we have that
\[
u_\Omega(x)-U_\Omega(x,z)\ge \tau-t\ge T+\frac{\widehat T}{2}- (T+\frac38\widehat T)=\frac18 \widehat T,
\]
which shows that $\Omega_{\tau}\setminus \Omega_{t,z}\subset\{|u_\Omega-U_\Omega(\cdot,z)|\ge \frac18\widehat T \}$. Hence, using Chebichev's inequality and {Lemma \ref{L2} with $\varphi=u_\Omega$}, it holds that
\[
\begin{aligned}
	\frac{|\Omega_{\tau}\setminus \Omega_{t,z}|}{|\Omega|} &\le\frac{|\{|u_\Omega-U_\Omega(\cdot,z)|\ge \frac18\widehat T \}|}{|\Omega|} \\ 
	&\le  \frac{64}{|\Omega|\widehat T^2}\|u_\Omega - U_\Omega(\cdot,z)\|^2_{L^2(\R^n)}\le \frac{64{c_{n,s}[u_\Omega]^2}}{|\Omega|\widehat T^2}z^{2s}\le\gamma \A(\Omega),
\end{aligned}
\]
as long as $z\le \left(\frac{\widehat{T}}{8}\sqrt{\frac{\gamma\mathcal{A}(\Omega) |\Omega|}{ {c}_{n,s}{\Cs(\Omega)}}}\right)^\frac{1}{s}$. Similarly, one can check that 
\[
\frac{|\Omega_{t,z}\setminus\Omega_{T+\frac{\widehat{T}}{16}}|}{|\Omega|}\le \gamma\A(\Omega).
\]
as long as $z\le \left(\frac{\widehat{T}}{16}\sqrt{\frac{\gamma\mathcal{A}(\Omega) |\Omega|}{ {c}_{n,s}{\Cs(\Omega)}}}\right)^\frac{1}{s}$.
This, together with \eqref{smiths} entails that 
\[
{\frac{||\Omega_{t,z}|-|\Omega||}{|\Omega|}\leq} 3\gamma\A(\Omega),
\]
{which proves \eqref{ok}.
Finally, applying Lemma \ref{trasfasimmetria} (with $\delta=3\gamma$ and $\gamma$ chosen to be in $(0,1/6)$), we deduce \eqref{asimmetria}.}
\end{proof}

We can now give the proof of our main result.

\begin{proof}[Proof of Theorem \ref{main}]
By scaling invariance, we may assume $|\Omega|=1$. {We also fix $\gamma=10^{-1}\in (0,1/9)$ throughout the proof.} We have to prove that
\begin{equation}\label{wanted}\Cs(\Omega)-\Cs(B)\ge C{\Cs(B)}\mathcal A(\Omega)^{\frac{3}{s}},\end{equation}
for some $C>0$, where $B$ is a ball with unit volume.

{We start by observing that if $\Cs(\Omega)>2\Cs(B)$, then, since $\mathcal A(\Omega)\leq 2$, we easily deduce that
$$\Cs(\Omega)-\Cs(B) >\Cs(B) =\frac{\Cs(B)}{2^{\frac{3}{s}}}2^{\frac{3}{s}} \geq\frac{\Cs(B)}{2^{\frac{3}{s}}}\mathcal A(\Omega)^{\frac{3}{s}},$$
which proves \eqref{wanted} with ${C= 2^{-\frac{3}{s}}}$.

Thus, we can just consider the case in which 
\begin{equation}\label{case-ok}
\Cs(\Omega)\le2\Cs(B).
\end{equation}
}
%
%
We recall that in \eqref{T}, we have defined the level $T$ as
$$T=\inf\left\{0\leq t\leq 1\colon \mu(t) \le |\Omega|\left(1+\gamma\mathcal A(\Omega)\right)\right\}.$$
Notice that by Lemma \ref{trasfasimmetria} as long as $1\ge t\ge T$ it holds
\[
\A(\Omega_{t})\ge c_\gamma\A(\Omega),
\]
for some $c_\gamma$ independent of $t$ and $\Omega$.


{We now distinguish between two cases, in terms of the relation of $T$ with $\mathcal{A}(\Omega)$. More precisely, we let
\begin{equation}\label{eq:def_kappa}
	\lambda:=\frac{n-2s}{n}\gamma\quad\text{and}\quad \kappa:=\frac{\lambda}{4(1+2\lambda)}.
\end{equation}
We consider the ranges
\begin{equation*}
	1-T\geq \kappa\mathcal{A}(\Omega)\quad\text{and}\quad 1-T<\kappa\A(\Omega).
\end{equation*}}

\noindent \underline{Case $1-T \ge \mathcal {\kappa}A(\Omega)$}. In this case we argue  as in the proof of Proposition 4.4 and of Theorem 1.3 (case $T>T_0$) in \cite{BCV}. The idea consists in introducing quantitative elements in the proof of the P\'olya-Szeg\"{o} inequality by applying the quantitative isoperimetric inequality on each (horizontal) level set $\Omega_{t,z}$ of the function $U_\Omega(\cdot, z)$. 

First, we recall that
\begin{equation}
\label{conto0}
\Cs(\Omega)=[u_\Omega]_{\mathcal{D}^{s,2}(\mathbb{R}^n)}^2
=\alpha_{n,s}^{-1}\,\int_{\R^{n+1}_+} z^{1-2\,s}\,|\nabla_x U_\Omega|^2\,dx\,dz
+\alpha_{n,s}^{-1}\,\int_{\R^{n+1}_+} z^{1-2\,s}\,\left|\partial_z U_\Omega\right|^2\,dx\,dz.
\end{equation}
For what concerns the $z$-derivative, from \eqref{polyaszegoz} we know that
\[
\int_{\R^{n+1}_+} z^{1-2\,s}\,\left|\partial_z U_\Omega\right|^2\,dx\,dz\ge \int_{\R^{n+1}_+} z^{1-2\,s}\,\left|\partial_z U^*_{\Omega}\right|^2\,dx\,dz.
\]
For the $x$-derivative, we argue as in the local case. By the coarea formula, we have
\begin{equation}
\begin{aligned}
\label{conto1}
&\int_{\R^{n+1}_+} z^{1-2\,s}\,|\nabla_x U_\Omega|^2\,dx\,dz\\
&\hspace{2em}=\int_{0}^{+\infty}z^{1-2s}\left(\int_{0}^{+\infty}\left(\int_{\{x\in\mathbb{R}^n \, : \, U_\Omega(x,z)=t\}} \,|\nabla_x U_\Omega|^2\,\frac{d\mathcal{H}^{n-1}(x)}{|\nabla_x U_\Omega|}\right)\,dt\right)\,dz\\
&\hspace{2em}\ge \int_{0}^{+\infty}z^{1-2s}\left(\int_{0}^{+\infty}\frac{P(\Omega_{t,z})^2}{\displaystyle\int_{\{x\in\mathbb{R}^n \, : \, U_\Omega(x,z)=t\}}\frac{d\mathcal{H}^{n-1}(x)}{|\nabla_x U_\Omega|}}\,dt\right)\,dz
\end{aligned}
\end{equation}
where $P(\Omega_{t,z})$ denotes the perimeter of the set $\Omega_{t,z}$, and in the last step we have used Jensen's inequality. Using the quantitative isoperimetric inequality, one can prove that 
\begin{equation}\label{eq:isoperim}
P(\Omega_{t,z})^2\geq P(\Omega_{t,z}^*)^2+c_n\mu_z(t)^{\frac{2(n-1)}{n}}\mathcal{A}(\Omega_{t,z})^2,
\end{equation}
where 
\[
\Omega_{t,z}^*:=\{x\in\R^n\colon  U_\Omega^*(x,z)\geq t\},
\]
and
\begin{equation*}
	c_n:=2\omega_n^{\frac{2}{n}}\frac{(2-2^{\frac{n-1}{n}})^3}{(181)^2n^{12}}.
\end{equation*}
 {The proof of \eqref{eq:isoperim} can be easily carried out by following \cite[Lemma 2.9]{Brasco2017}, see also the proof of Proposition 4.4 in \cite{BCV}. By definition of symmetric rearrangement, we have that
\[
	\mu_z(t)=\abs{\Omega_{t,z}^*}\quad\text{for a.e. }t\in (0,1)
\]
and, from Lemma 3.2 and inequality $(3.19)$ in \cite{Cianchi2002} (see also $(2.6)$ in \cite{FMP08}), we know that
\[
	-\mu_z'(t)=\int_{\{x\in\mathbb{R}^n \colon U_\Omega^*(x,z)=t\}}\frac{d\mathcal{H}^{n-1}(x)}{|\nabla_x U_\Omega^*|}\geq \int_{\{x\in\mathbb{R}^n \, : \, U_\Omega(x,z)=t\}}\frac{d\mathcal{H}^{n-1}(x)}{|\nabla_x U_\Omega|}.
\]
Therefore, combining \eqref{conto1} and \eqref{eq:isoperim} with this last inequality, we can estimate the $L^2$-norm of the $x$-gradient as follows}
\[
\begin{split}
\int_{\mathbb{R}^{n+1}_+} z^{1-2\,s}\,|\nabla_x U_\Omega|^2\,dx\,dz
&\ge \int_{0}^{+\infty}z^{1-2s}\left(\int_{0}^{+\infty}\frac{P(\Omega_{t,z}^*)^2}{-\mu_z'(t)}\,dt\right)\,dz\\
&+c_n\int_{0}^{+\infty} z^{1-2\,s}\,\left(\int_0^{+\infty}\frac{\left(\mu_z(t)^\frac{n-1}{n}\right)^2\,\mathcal{A}(\Omega_{t,z})^2}{-\mu'_z(t)} dt\right)\,dz\\
& = \int_{\mathbb{R}^{n+1}_+} z^{1-2\,s}\,|\nabla_x U^*_\Omega|^2\,dx\,dz \\
&+c_n\int_{0}^{+\infty} z^{1-2\,s}\,\left(\int_0^{+\infty}\frac{\left(\mu_z(t)^\frac{n-1}{n}\right)^2\,\mathcal{A}(\Omega_{t,z})^2}{-\mu'_z(t)} dt\right)\,dz.\\
\end{split}
\]
Moreover, we have seen in the proof of Theorem \ref{PS} that
\[
\alpha_{n,s}^{-1}\,\int_{\R^{n+1}_+} z^{1-2\,s}\,|\nabla U^*_\Omega|^2\,dx\,dz\geq\Cs(B).
\]
Collecting all together, we obtain
 \[
\begin{split}
\Cs(\Omega)&=\alpha_{n,s}^{-1}\,\int_{\R^{n+1}_+} z^{1-2\,s}\,|\nabla_x U_\Omega|^2\,dx\,dz
+\alpha_{n,s}^{-1}\,\int_{\R^{n+1}_+} z^{1-2\,s}\,\left|\partial_z U_\Omega\right|^2\,dx\,dz\\
&\geq\Cs(B)+{\alpha_{n,s}^{-1}}c_n\int_{0}^{+\infty} z^{1-2\,s}\,\left(\int_0^{+\infty}\frac{\left(\mu_z(t)^\frac{n-1}{n}\right)^2\,\mathcal{A}(\Omega_{t,z})^2}{-\mu'_z(t)} dt\right)\,dz\\
\end{split}
\]
 We use now Proposition \ref{prop:asimmetrie}, to pass from $\mathcal{A}(\Omega_{t,z})$ to $\mathcal{A}(\Omega)$. {Let $z_0=\left(\frac{\widehat{T}}{16}\sqrt{\frac{\gamma\mathcal{A}(\Omega) |\Omega|}{ {c}_{n,s}{\Cs(\Omega)}}}\right)^\frac{1}{s}$ be as in Proposition \ref{prop:asimmetrie}}.  Set also $\widehat T=1-T$.
We have
 \[
\begin{split}
\Cs(\Omega)-\Cs(B)&\ge {\alpha_{n,s}^{-1}}c_n\int_{0}^{+\infty} z^{1-2\,s}\,\left(\int_{0}^{+\infty}\frac{\left(\mu_z(t)^\frac{n-1}{n}\right)^2\,\mathcal{A}(\Omega_{t,z})^2}{-\mu'_z(t)} dt\right)\,dz\\
&\ge {\alpha_{n,s}^{-1}}c_n\,\int_{0}^{z_0} z^{1-2\,s}\,\left( \int_{{T+\frac{\widehat T}{8}}}^{T+\frac{3}{8}\widehat T} \mathcal{A}(\Omega_{t,z})^2\,\frac{\left(\mu_{z}(t)^\frac{n-1}{n}\right)^2}{-\mu'_z(t)}\,dt\right)\,dz\\
&\ge{\alpha_{n,s}^{-1}} c_n\,c_\gamma^{{2}}\,\mathcal{A}(\Omega)^2\,\int_{0}^{z_0} z^{1-2\,s}\,\left( \int_{{T+\frac{\widehat T}{8}}}^{T+\frac{3}{8}\widehat T}\frac{\left(\mu_{z}(t)^\frac{n-1}{n}\right)^2}{-\mu'_z(t)}\,dt\right)\,dz,
\end{split}
\]
where, in the last inequality, we used \eqref{asimmetria}.
{It remains to estimate the term
\begin{equation}\label{last}
\int_{0}^{z_0} z^{1-2\,s}\int_{T+\frac{\widehat T}{8}}^{T+\frac{3}{8}\widehat T}\frac{\left(\mu_{z}(t)^\frac{n-1}{n}\right)^2}{-\mu'_z(t)}\,dt\,dz.
\end{equation}
First, we observe that, since $\gamma<1/9$, using \eqref{ok} and the fact that $\mathcal A(\Omega) \leq2$, we have
$$\mu_z(t)\ge 1-3\gamma \mathcal A(\Omega)\geq \frac{1}{3}.$$
Hence, in order to estimate \eqref{last}, it is enough to control from below the quantity
$$\int_{0}^{z_0} z^{1-2\,s}\int_{T+\frac{\widehat T}{8}}^{T+\frac{3}{8}\widehat T}\frac{1}{-\mu'_z(t)}\,dt\,dz.$$
By Jensen inequality and recalling the definition of $\mu_z$, we have
\[\int_{T+\frac{\widehat T}{8}}^{T+\frac{3}{8}\widehat T}\frac{1}{-\mu'_z(t)}\,dt\ge \frac{\widehat T^2}{16} \frac{1}{\int_{T+\frac{\widehat T}{8}}^{T+\frac{3}{8}\widehat T}-\mu_z'(t)\,dt} \ge \frac{\widehat T^2}{16}\frac{1}{|\Omega_{T+\frac{\widehat T}{8},z}|-|\Omega_{T+\frac{3}{8}\widehat T,z}|}.\]
Using again \eqref{ok}, we deduce that
$$|\Omega_{T+\frac{\widehat T}{8},z}|-|\Omega_{T+\frac{3}{8}\widehat T,z}|\le 1+3\gamma \mathcal A(\Omega) - {(1-3\gamma\mathcal{A}(\Omega))}={6}\gamma \mathcal A(\Omega).$$
{We can now estimate \eqref{last} as follows
\begin{equation*}
	\int_{0}^{z_0} z^{1-2\,s}\int_{T+\frac{\widehat T}{8}}^{T+\frac{3}{8}\widehat T}\frac{\left(\mu_{z}(t)^\frac{n-1}{n}\right)^2}{-\mu'_z(t)}\,dt\,dz\geq C_1 \frac{\widehat{T}^2}{\mathcal{A}(\Omega)}\int_{0}^{z_0} z^{1-2\,s}\,dz,
\end{equation*}
where $C_1>0$ depends only on $n$ and $\gamma$. This, in turn, yields the following estimate for the capacity variation
\begin{equation*}
	\Cs(\Omega)-\Cs(B)\geq \frac{C_2}{\alpha_{n,s}}\mathcal{A}(\Omega)\widehat{T}^2 \int_{0}^{z_0} z^{1-2\,s}\,dz,
\end{equation*}
with $C_2>0$ depending only on $n$ and $\gamma$.
Finally, recalling the definition of $z_0$ (in Proposition \ref{prop:asimmetrie}) and that we are in the ranges $\Cs(\Omega)\leq 2\Cs(B)$ and $\widehat T=1-T\ge \mathcal \kappa \A(\Omega)$, we obtain 
\begin{align*}
	\Cs(\Omega)-\Cs(B)&\geq \frac{C_3}{\alpha_{n,s}}\left(\frac{C_4}{c_{n,s}} \right)^{\frac{1}{s}-1}\frac{1}{(1-s)\Cs(B)^{\frac{1}{s}}}\Cs(B)\mathcal{A}(\Omega)^{\frac{1}{s}}\widehat{T}^{\frac{2}{s}} \\
	&\geq \frac{\kappa^{\frac{2}{s}}C_3}{\alpha_{n,s}}\left(\frac{C_4}{c_{n,s}}\right)^{\frac{1}{s}-1}\frac{1}{(1-s)\Cs(B)^{\frac{1}{s}}}\Cs(B)\mathcal{A}(\Omega)^{\frac{3}{s}},
\end{align*}
with $C_3,C_4>0$ depending on $n$ and $\gamma$. In particular, given $\gamma=10^{-1}$ it is possible to see that
\begin{equation}\label{eq:C_3}
	C_3=\frac{5}{3^5(181)^2}\frac{(3\omega_n)^{\frac{2}{n}}}{n^{12}}(1-2^{-\frac{1}{n}})^3.
\end{equation}
 Therefore \eqref{wanted} holds with
\begin{equation*}
	C=\frac{\kappa^{\frac{2}{s}}C_3}{\alpha_{n,s}}\left(\frac{C_4}{c_{n,s}}\right)^{\frac{1}{s}-1}\frac{1}{(1-s)\Cs(B)^{\frac{1}{s}}}.
\end{equation*}
}

}

\noindent \underline{Case $1-T< {\kappa} \A(\Omega)$} 
In this case we get the quantitative inequality by means of a suitable  test  function for the definition of $\Cs$.
Let us define
$$w_T=\min\left\{1,\frac{u_\Omega}{T}\right\}.$$
{It is possible to see that $w_T$ is an admissible competitor for $\Cs(\Omega_T)$. Indeed, let $\xi_k:=\tilde{\xi}_k\star \rho_{\frac{1}{2k}}$, where $\rho_\epsilon$ denotes a standard mollifier in $\mathbb{R}$ and $\tilde{\xi}_k\colon \mathbb{R}\to \mathbb{R}$ is defined as follows
\[
\tilde{\xi}_k(\sigma):=\begin{cases}
	0, &\text{if }\sigma\leq \frac{1}{k}, \\
	\frac{k}{k-2}\left(\sigma-\frac{1}{k}\right),&\text{if }\frac{1}{k}\leq \sigma\leq 1-\frac{1}{k}, \\
	1,&\text{if }\sigma\geq 1-\frac{1}{k}.
\end{cases}
\]
It is clear that $\xi_k\in C^\infty(\mathbb{R})$ and that
\[
\xi_k(\sigma)\to \min\{ \sigma^+,1 \},\quad\text{uniformly as }k\to\infty.
\]
Moreover, if we let $w_k:=\xi_k\circ w_T$, one can see that $w_k\in\mathcal{D}^{s,2}(\mathbb{R}^n)$ and, thanks to the continuity of $w_T$, that $w_k=1$ in an open $\Omega_k\supseteq \Omega_T$, thus implying that $w_k-\eta_{\Omega_T}\in\mathcal{D}^{s,2}(\mathbb{R}^n\setminus\Omega_T)$, being $\eta_{\Omega_T}\in C_c^\infty(\mathbb{R}^n)$ such that $\eta_{\Omega_T}=1$ in a neighbourhood of $\Omega_T$. Finally, it is easy to check that $w_k\to w_T$ in $\mathcal{D}^{s,2}(\mathbb{R}^n)$ as $k\to \infty$, which means that $w_T-\eta_{\Omega_T}\in\mathcal{D}^{s,2}(\mathbb{R}^n\setminus\Omega_T)$.}
{Therefore, we have
\begin{equation*}
\Cs(\Omega_T)\le [w_T]^2_s\le \frac{1}{T^2} [u_{\Omega}]^2_s=\frac{1}{T^2}\Cs(\Omega).
\end{equation*}
From the previous inequality, the isocapacitary inequality \eqref{eq:frac_iso} and \eqref{eq:mu_T} we obtain that
\begin{equation}\label{eq:easy_1}
	\Cs(\Omega)\ge T^2\Cs(B)|\Omega_T|^{\frac{n-2s}{n}}\geq T^2\Cs(B)\left(1+\gamma\mathcal A(\Omega)\right)^{\frac{n-2s}{n}}.
\end{equation}
By convexity, we know that
\begin{equation}\label{eq:easy_2}
	T^2\geq 1-2(1-T)\geq 1-2\kappa\mathcal{A}(\Omega)
\end{equation}
and that
\begin{equation}\label{eq:easy_3}
	(1+\gamma\A(\Omega))^{\frac{n-2s}{n}}\geq 1+\lambda \A(\Omega),
\end{equation}
with $\lambda$ as in \eqref{eq:def_kappa}. By the definition of $\lambda$ and $\kappa$, as in \eqref{eq:def_kappa}, we derive that
\begin{equation}\label{eq:easy_4}
	(1-2\kappa \A(\Omega))(1+\lambda\A(\Omega))\geq 1+\frac{\lambda}{2}\A(\Omega)\geq (1+C_5\A(\Omega)^{\frac{3}{s}}),
\end{equation}
with
\[
C_5:=2^{-\frac{3}{s}}\lambda
\]
Putting together \eqref{eq:easy_1} with \eqref{eq:easy_2}, \eqref{eq:easy_3} and \eqref{eq:easy_4}, we obtain \eqref{wanted} with $C=C_5$, thus concluding the proof.
}

\end{proof}

{\begin{remark}\label{rmk:asympt_1}
	By carefully scanning the proof of Theorem \ref{main}, one can explicitly find the constant $C_{n,s}$ appearing in \eqref{quantitativa}, which amounts to
	\begin{equation*}
		C_{n,s}=\max\left\{ 2^{-\frac{3}{s}},\frac{\kappa^{\frac{2}{s}}C_3}{\alpha_{n,s}}\left(\frac{C_4}{c_{n,s}}\right)^{\frac{1}{s}-1}\frac{1}{(1-s)\Cs(B)^{\frac{1}{s}}}  \right\},
	\end{equation*}
	where $\kappa$ is as in \eqref{eq:def_kappa}, and $C_3,C_4$ depend only on $n$ (indeed their dependence on $\gamma$ stated in the proof of Theorem \ref{main} is actually pointless, being $\gamma$ universally fixed in $(0,1/9)$, see \eqref{eq:C_3} for the explicit value of $C_3$). The constants $\alpha_{n,s}$ and $c_{n,s}$ are as in \eqref{eq:alpha} and \eqref{eq:c}, respectively, and they are uniformly bounded away from $0$ and $+\infty$ as $s\to 1^-$, see \cite[Remark 2.7]{BCV}. Moreover, it can be easily checked that this fact holds true for the constant $\kappa$ as well, by its definition, when $n\geq 3$.
\end{remark}}


\section{\texorpdfstring{Asymptotics as $s\nearrow 1$}{Asymptotics as s converges to 1}}
In this section we prove Thereom \ref{main2}.
We  recall  the definition of the standard (Newtonian) capacity of a closed $\Omega\sub\R^n$ {for $n\geq3$}, which is equivalent to \eqref{eq:def_std_cap} when $\Omega$ is a compact set
\begin{equation}\label{eq:def_cap}
	\C(\Omega)=\inf\left\{\int_{\R^n} \abs{\nabla u}^2\,dx\colon u\in \mathcal{D}^{1,2}(\R^n)~\text{and }u-\eta_\Omega\in \mathcal{D}^{1,2}(\R^n\setminus \Omega) \right\},
\end{equation}
where $\eta_\Omega\in C_c^\infty(\R^n)$ is such that $\eta_\Omega=1$ in an open neighbourhood of $\Omega$ and, for any open $\mathcal{O}\sub\R^n$, the space $\mathcal{D}^{1,2}(\mathcal{O})$ is defined as the completion of $C_c^\infty(\mathcal{O})$ with respect to the norm
\[
	u\mapsto \left(\int_{\mathcal{O}}\abs{\nabla u}^2\,dx \right)^{1/2}.
\]
We also recall that, when $\mathcal{O}$ is bounded, then the space $\mathcal{D}^{1,2}(\mathcal{O})$ coincide with the usual Sobolev space $W_0^{1,2}(\mathcal{O})$, thanks to the validity of the Poincar\'e inequality.



\begin{proof}[Proof of Proposition \ref{main2}]
	\textbf{Proof of \eqref{eq:asympt_s_th1}.}
	The first part of the statement follows easily by a celebrated result by Bourgain-Brezis-Mironescu stating  that
	\[
	\lim_{s\nearrow 1}(1-s)\,[\varphi]^2_{s}=\frac{\omega_n}{2}\,\int_{\R^n} |\nabla \varphi|^2\,dx,\qquad \mbox{ for every } \varphi\in C^\infty_c(\R^n).
	\]
	Indeed, by taking a function $\varphi\in C_c^\infty(\R^n)$ satisfying $\varphi\ge \chi_\Omega$, we deduce that 
	\[
	\limsup_{s\nearrow 1}(1-s)\,{\rm cap}_s(\Omega)\le \lim_{s\nearrow 1}(1-s)\,[\varphi]^2_{s}=\frac{\omega_n}{2}\,\int_{\R^n} |\nabla \varphi|^2\,dx.
	\]
	Finally, \eqref{eq:asympt_s_th1} follows by taking the infimum over all admissible $\varphi$.
	
	\noindent\textbf{Proof of \eqref{eq:asympt_s_th2}.} Let us fix $s_0\in(0,1)$ and let us denote by $u_{s,\Omega}$ the $s$-capacitary potential of the set $\Omega$.
	Since $\Omega$ is compact, there exists a ball $B_{R_0}$ which contains $\Omega$. Hence, we have that
	$$u_{s,\Omega} \le u_{s,B_{R_0}}\;\; \mbox{a.e. in }\R^n.$$
	This can be easily proved by taking the Kelvin transform of the above functions and applying the maximum principle as in \cite[Theorem 3.3.2]{BV2016}.	We can now take advantage of the following precise decay rate of $u_{s,B_{R_0}}$, established in \cite[Proposition 3.6]{BMS}:
	\begin{equation}\label{decay}
		u_{s,\Omega}(x) \le u_{s,B_{R_0}}(x)\le \frac{2{R_0^{n-2s}}}{|x|^{n-2s}},\quad \mbox{for } |x|>R_0.
	\end{equation}
	Let us define  an almost optimal function, given by a suitable truncation of $u_{s,\Omega}$.
	For any fixed $\varepsilon >0$, we set
	$$u_{s,\Omega}^{\varepsilon}:=\frac{(u_{s,\Omega}-\varepsilon)^+}{1-\varepsilon}.$$
	We claim that $u_{s,\Omega}^{\varepsilon}\in \Ds$ and $u_{s,\Omega}^{\varepsilon}-\eta_\Omega\in \DsO$, with $\eta_\Omega\in C_c^\infty(\R^n)$ being such that $\eta_\Omega=1$ in an open neighbourhood of $\Omega$. Indeed, since $u_{s,\Omega}-\eta_\Omega\in \DsO$, there exists a sequence $\{v_k\}_k\sub C_c^\infty(\R^n\setminus\Omega)$ such that $v_k\to u_{s,\Omega}-\eta_\Omega$ in $\Ds$ as $k\to\infty$. If we now let $u_k:=v_k+\eta_\Omega$ we have that $u_k\in C_c^\infty(\R^n)$ and $u_k=1$ in an open $\Omega_k\supseteq\Omega$. Now, if we consider the function
	\[
		u_k^\eps:=\frac{(u_k-\eps)^+}{1-\eps}
	\]
	we have that $u_k^\eps\in \Ds$ and $u_k^\eps-\eta_\Omega\in \DsO$. Moreover, $u_k^\eps\to u_{s,\Omega}^\varepsilon$ in $\Ds$ as $k\to\infty$, thus proving the claim. We also observe that the family $\{u_{s,\Omega}^{\varepsilon}\}_{s\in (s_0,1)}$ satisfies the following properties:
	\begin{enumerate}
		\item there exists $\bar R=\bar R(\varepsilon)>0$, depending only on $\varepsilon$, such that 
		$$\mbox{supp }u_{s,\Omega}^\varepsilon \subset B_{\bar R}\,\,\mbox{for any } s\in(s_0,1).$$
		This follows by the upper bound \eqref{decay}: we choose $\bar R>\left(\frac{2{R_0^{n-2s}}}{\varepsilon}\right)^{\frac{1}{n-2}}$ {with $R_0$ being such that $\Omega\sub B_{R_0}$}. In particular, this implies that $u_{s,\Omega}^\varepsilon\in \widetilde{W}_0^{s,2}(B_{\bar{R}})$, where, for any open $\mathcal{O}\sub\R^n$ we denote
		\begin{equation*}
			\widetilde{W}_0^{s,2}(\mathcal{O}):=\left\{ u\in L^1_{\textup{loc}}(\R^n)\colon [u]_s<\infty~\text{and }u=0~\text{in }\R^n\setminus\mathcal{O} \right\},
		\end{equation*}
		which, in case $\mathcal{O}$ is bounded and has Lipschitz boundary, coincides with the space $\mathcal{D}^{s,2}(\mathcal{O})$, see \cite[Proposition B.1]{BPS};
		\item there holds
		\begin{equation}\label{eq:u_s_eps}
			(1-s)[u_{s,\Omega}^{\varepsilon}]_s^2 \leq (1-s)\frac{[u_{s,\Omega}]_s^2}{(1-\varepsilon)^2}\le C_1,
		\end{equation}
		for any $\varepsilon>0$ and $s\in (s_0,1)$, with $C_1>0$ independent of $\varepsilon$ and $s$. This is a direct consequence of \eqref{eq:asympt_s_th1}.
	\end{enumerate}
	Hence we can apply \cite[Proposition 3.6]{BPS} to the family $\{u_{s,\Omega}^{\varepsilon}\}_{s\in(s_0,1)}$ to deduce that 
	there exists an increasing sequence $s_k \in (s_0, 1)$ converging to $1$ and a function $u_{\Omega}^\varepsilon\in W^{1,2}_0(B_{\bar R})$
	such that
	$$\lim_{k\rightarrow \infty} \|u_{s_k,\Omega}^{\varepsilon}-u_{\Omega}^\varepsilon\|_{L^2(B_{\bar R})}=0.$$
	Analogously, being $\Omega$ a Lipschitz domain, we know that $u_{s,\Omega}^{\varepsilon}-\eta_\Omega\in \widetilde{W}_0^{s,2}(B_{\bar{R}}\setminus\Omega)$ and that
	\[
		(1-s)[u_{s,\Omega}^{\varepsilon}-\eta_\Omega]_s^2\leq C_2
	\]
	for all $\varepsilon>0$ and $s\in (s_0,1)$, with $C_2>0$ independent of $\varepsilon$ and $s$. Therefore, we can apply \cite[Proposition 3.6]{BPS} to the family $\{u_{s,\Omega}^{\varepsilon}-\eta_\Omega\}_{s\in(s_0,1)}$ as well, and this entails the existence of a (not relabeled) subsequence $s_k\in(s_0,1)$ and of a function $v_{\Omega}^\eps\in W_0^{1,2}(B_{\bar{R}}\setminus\Omega)$ such that
	\[
		\lim_{k\rightarrow \infty} \|u_{s_k,\Omega}^{\varepsilon}-\eta_\Omega-v_{\Omega}^\varepsilon\|_{L^2(B_{\bar R})}=0,
	\]
	where the functions are trivially extended in $\Omega$. As a consequence, we obtain that
	\[
		u_\Omega^\eps-\eta_\Omega=v_\Omega^\eps\in W_0^{1,2}(B_{\bar{R}}\setminus \Omega),
	\]
	which, in turn, implies that the trivial extension of $u_\Omega^\eps$ to the whole $\R^n$ is an admissible competitor for ${\mathrm{cap}}(\Omega)$.	Hence we have
	\begin{align*}
		\frac{\omega_n}{2}{\rm{cap}}(\Omega) \leq \frac{\omega_n}{2}\int_{\R^n}|\nabla u^\varepsilon_{\Omega}|^2\,dx &\leq \liminf_{s\nearrow 1}(1-s)[u^\varepsilon_{s,\Omega}]_s^2\\
		&\leq \frac{1}{(1-\varepsilon)^2}\liminf_{s\nearrow 1}(1-s)[u_{s,\Omega}]_s^2,
	\end{align*}
	where in the second inequality we have used the $\Gamma$-convergence result by Brasco, Parini, Squassina (more precisely, Proposition 3.11 in \cite{BPS}) and in the last one \eqref{eq:u_s_eps}. Finally, we conclude by letting $\varepsilon \rightarrow 0$. 
\end{proof}

\begin{remark}\label{rmk:asympt_2}
	Thanks to Theorem \ref{main2} it is possible to explicitly compute the limit as $s\to 1^-$ of the constant $C_{n,s}$ as in Theorem \ref{main}, which coincides with the constant $C_n$ appearing in Corollary \ref{cor}. Indeed, from the definitions of $\alpha_{n,s}$ and $c_{n,s}$, given in \eqref{eq:alpha} and \eqref{eq:c} respectively, and the property of the Gamma function, it is easy to see that
	\[
		\lim_{s\to 1^-}\alpha_{n,s}=\lim_{s\to 1^-}c_{n,s}=\pi^{-\frac{n}{2}}\Gamma\left(\frac{n+2}{2}\right).
	\]
	Moreover, if $B$ denotes the unitary ball in $\R^n$, in view of Theorem \ref{main2} we have that
	\[
		\lim_{s\to 1^-}(1-s)\Cs(B)=\frac{\omega_n}{2}\C(B)=\frac{n(n-2)}{2}\omega_n^2,
	\]
	see e.g. \cite[Theorem 2.8, point $(i)$]{Maly1997} for the explicit value of $\C(B)$. Thanks to these facts, Remark \ref{rmk:asympt_1} and basic calculus, it is easy to see that
	\[
		C_n=\lim_{s\to 1^-}C_{n,s}=\max\left\{ 2^{-3}, \frac{\kappa_1^2 C_3\pi^{\frac{n}{2}}}{\Gamma\left(\frac{n+2}{2}\right)}\frac{2}{n(n-2)\omega_n^2} \right\},
	\]
	where
	\[
		\kappa_1:=\frac{\lambda_1}{4(1+2\lambda_1)},\quad\text{with}\quad\lambda_1:=\frac{n-2}{10n},
	\]
	and $C_3$ as in \eqref{eq:C_3}.
\end{remark}

\section*{Acknowledgments}
E. Cinti is partially supported by MINECO grant MTM2017-84214-C2-1-P, by Gruppo Nazionale per l’Analisi
Matematica, la Probabilità e le loro Applicazioni (GNAMPA) of the Istituto
Nazionale di Alta Matematica (INdAM),
and is part of the Catalan research group 2017 SGR 1392.  B. Ruffini acknowledges partial support from the ANR-18-CE40-0013 SHAPO financed by the French Agence Nationale de la Recherche (ANR).  R. Ognibene acknowledges support from the MIUR-PRIN project No. 2017TEXA3H. A special thank goes to L. Brasco for several fruitful discussions and suggestions. {Eventually, the authors  thank Italy national football team for providing a joyful atmosphere during the last stages of the drafting process of this work.}

\bibliographystyle{aomalpha}

\providecommand{\bysame}{\leavevmode\hbox to3em{\hrulefill}\thinspace}
\providecommand{\noopsort}[1]{}
\providecommand{\mr}[1]{\href{http://www.ams.org/mathscinet-getitem?mr=#1}{MR~#1}}
\providecommand{\zbl}[1]{\href{http://www.zentralblatt-math.org/zmath/en/search/?q=an:#1}{Zbl~#1}}
\providecommand{\jfm}[1]{\href{http://www.emis.de/cgi-bin/JFM-item?#1}{JFM~#1}}
\providecommand{\arxiv}[1]{\href{http://www.arxiv.org/abs/#1}{arXiv~#1}}
\providecommand{\doi}[1]{\url{https://doi.org/#1}}
\providecommand{\MR}{\relax\ifhmode\unskip\space\fi MR }
\providecommand{\MRhref}[2]{%
	\href{http://www.ams.org/mathscinet-getitem?mr=#1}{#2}
}
\providecommand{\href}[2]{#2}


\end{document}